\documentclass[3p,authoryear,times,final]{elsarticle}
\usepackage{latexsym,amssymb,amsmath,amsfonts,amsthm,tikz,verbatim,empheq}
\usetikzlibrary{shapes.geometric}
\usepackage[authoryear]{natbib}

\RequirePackage[colorlinks,citecolor=blue,urlcolor=blue]{hyperref}

\newtheorem{theorem}{Theorem}

\newtheorem{lemma}{Lemma}[section]

\newtheorem{conjecture}{Conjecture}

\def\beq{ \begin{equation} }
\def\eeq{ \end{equation} }

\def\ep{\epsilon}

\def\square{\vcenter{\vbox{\hrule height .4pt
  \hbox{\vrule width .4pt height 5pt \kern 5pt
        \vrule width .4pt} \hrule height .4pt}}}

\def\RR{\mathbb{R}}
\def\ZZ{\mathbb{Z}}

\begin{document}
\begin{frontmatter}

\title{Phase transitions for a planar quadratic contact process}
\author{Mariya Bessonov\fnref{MB}}
\address{NYC College of Technology, Department of Mathematics, 
300 Jay Street, Brooklyn, NY 11201}
\ead{mbessonov@citytech.cuny.edu}
\author{Richard Durrett\fnref{RD}}
\address{Duke University, Department of Mathematics,
Box 90320, Durham, NC 27708}
\ead{rtd@math.duke.edu}

\fntext[MB]{Partially supported by NSF grant DMS-1515800 and PSC-CUNY grant 67095-00 45.}
\fntext[RD]{Partially supported by NSF grant DMS-1005470.}


\begin{abstract}
We study a two dimensional version of Neuhauser's long range sexual reproduction model and prove results
that give bounds on the critical values $\lambda_f$ for the process to survive from a finite set
and $\lambda_e$ for the existence of a nontrivial stationary distribution. Our first result comes from a standard
block construction, while the second involves a comparison with the ``generic population model'' of  \cite{BG-91}. An interesting new feature of our work is the suggestion that, as in the one dimensional contact
process, edge speeds characterize critical values. We are able to prove the following for our quadratic contact process
when the range is large but suspect they are true for two dimensional finite range attractive particle systems that are symmetric with respect
to reflection in each axis. There is a speed $c(\theta)$ for the expansion of the process in each direction.
If $c(\theta) > 0$ in all directions, then $\lambda > \lambda_f$, while if at least one speed is positive, then
$\lambda > \lambda_e$. It is a challenging open problem to show that
if some speed is negative, then the system dies out from any finite set.
\end{abstract}
\end{frontmatter}

\section{Introduction}

The contact process introduced by \cite{Harris-74}, is perhaps the simplest spatial 
model for the spread of a species. A site in $\ZZ^d$ can be occupied ($\xi_t(x)=1$) or vacant ($\xi_t(x)=0$).
Occupied sites become vacant at rate 1, while vacant sites become occupied at rate $\lambda$
times the number of neighbors that are occupied, i.e., the death rate is constant and the birth rate is linear.  
All 0's is an absorbing state. Due to monotonicity,
if we let $\xi^1_t$ be the state of the process at time $t$ when we start with all sites occupied
then $\xi^1_t$ converges to a limit $\xi^1_\infty$, called the upper invariant measure,
which is the largest possible stationary distribution.
For this and the other facts about the contact process, see \citep{Liggett-99}.

Let $\xi^A_t$ be the contact process started with 1's on $A$ and 0 otherwise.
In principle, the contact process could have two critical values:
\begin{itemize}
\item
$\lambda_e = \inf \{ \lambda : \lim_{t\to\infty} P( \xi^1_t(x) = 1 ) > 0 \}$,  
the critical value for the existence of a nontrivial stationary distribution.
\item
$\lambda_f = \inf \{ \lambda : P( \xi^A_t \neq \varnothing \hbox{ for all $t$ }) > 0 \hbox{ for some finite set $A$} \}$.
\end{itemize}
Results of \cite{BezGr-94} imply that for finite range attractive processes,
$\lambda_e \le \lambda_f$. In the case of the contact process, self-duality
$$
P\left( \xi^{\{0\}}_t \not\equiv 0 \right) = P\left( \xi^1_t(0) = 1 \right)
$$
implies that the critical values coincide: $\lambda_e = \lambda_f$, a value we will call $\lambda_c$. 
Let $$\tau^A = \inf \left\{ t : \xi^A_t \equiv 0 \right\}$$ be the time at which the process dies out.
When $\lambda = \lambda_c$, $P\left( \tau^A <\infty \right) = 1$
for all finite sets $A$ and $\xi^1_\infty = \delta_0$. From this, it follows fairly easily that
the phase transition is continuous, i.e., $P\left( \xi^1_\infty(x) = 1\right)$ is a continuous function of $\lambda$.
If $\lambda > \lambda_c$, the complete convergence theorem implies that as $t\to\infty$,
$$
\xi^A_t \Rightarrow P\left( \tau^A < \infty\right) \delta_0 + P\left( \tau^A = \infty\right) \xi^1_\infty.
$$
As a consequence of this result, all stationary distributions have the form
$$\theta \delta_0 + (1-\theta) \xi^1_\infty.$$ The distribution $\delta_0$ is stationary, but is
unstable under perturbation. If we add spontaneous births at rate $\beta$, then there is a unique stationary 
distribution $\xi^{\beta}_\infty$ for the modified contact process, and as $\beta \downarrow 0$ we have
$$P( \xi^\beta_\infty(x) = 1 ) \downarrow P( \xi^1_\infty(x) = 1).$$

In this paper, we will consider a variant of the contact process that has sexual reproduction,
i.e., two individuals are needed to produce a new one. Many processes of this type have been
studied, but many open problems remain. To emphasize that these processes are an important and
natural generalization of the contact process with linear birth rates, we will call them 
quadratic contact processes. We use this term somewhat loosely. In the models we discuss, the birth rate will not always
be a quadratic function of the number of occupied neighbors.

The oldest ``quadratic contact process'' is {\it Toom's NEC model}, see \citep{Toom-74, Toom-80}. However, in its initial
formulation, it was a system in which each site could be in one of two states $\xi_n(x) \in \{ 1, -1\}$. 
Let $\zeta_{n}(x)$ be the majority opinion among $x, x+(0,1), x+(1,0)$ at time $n$. If $\zeta_{n}(x)=+1$, then
$\xi_{n+1}=+1$ with probability $1-p$ and $\xi_{n+1}(x)=-1$ with probability $p$. If $\zeta_n(x)=-1$, then
$\xi_{n+1}=-1$ with probability $1-q$ and $\xi_{n+1}(x)=1$ with probability $q$. Toom proved that
for small $p$ and $q$, there are two nontrivial stationary distributions. Note that in contrast to the Ising
model, there is nonergodicity even when the system is asymmetric. For more on this process, see
\citep{BeGr-85,HJG-90,MSTG-05}. 

Durrett and Gray~\cite{DuGr-85} reformulated Toom's process as a continuous time growth model in which
\begin{align*}
1 \to 0 & \quad\hbox{rate 1} \\
0 \to 1 & \quad\hbox{rate $\lambda$ if $\xi_t(x+(1,0))=\xi_t(x+(0,1))=1$}
\end{align*}
If the initial configuration for this process has no 1's outside a square, then there will never be any births outside the square, so $\lambda_f=\infty$.  
Using the contour method, they proved (in an unpublished work \citep{DuGr-85}, see an announcement of results in \citep{D-85})
\begin{enumerate}[(DG1)]
  \item $\lambda_e \le 110$.
\item If $p < p^* = 1 - p_c$, where $p_c$ is the critical value for oriented bond percolation in $d=2$, then the process starting from product measure with density $p$ dies out. 
\item If $\lambda > \lambda_e$ and $\beta$ is such that $6 \beta^{1/4} \lambda^{3/4} < 1$ then when we add spontaneous births at rate $\beta$ there are two stationary distributions.
\end{enumerate}
(DG2) is easy to prove. If there is an infinite path of 0's starting from a site $x$ in which each step is up or right, then these 0's can never become 1's. Many such paths exist when the density of 0's exceeds $p_c$ for oriented percolation. Indeed, the set of 0's on the line $x+y=k$ will be a discrete time contact process. Starting from a single 0 then with positive probability the discrete time contact process does not die out and when it does not die out the left and right edges grow linearly. Thus the up-right paths will be so numerous that the 1's are trapped in finite sets in between the permanent 0's  and they will die out. The interest in this result is that it shows that the complete convergence theorem is false. The fact that product measures with density $p<p^*$ converge to $\delta_0$ for any $\lambda < \infty$ suggests that this model has a discontinuous phase transition, but proving this is a difficult problem. 

Recently,  \cite{CD-13} have shown that the 
discrete time threshold two contact process on a random $r$-regular 
graph or on a homogeneous tree in which each vertex has degree $r$ 
has a discontinuous phase transition if the degree $r\ge 3$. 
 \cite{VD-13} have used simulation and 
heuristic arguments to study two versions of the quadratic 
contact process on random graphs generated by the configuration model. 
In their vertex centered case, which is the one relevant to this paper, 
they find a discontinuous transition for the Erd\"os-Renyi random graph, 
but on power law random graphs with degree distribution $p_k \sim Ck^{-\alpha}$ and $\alpha < 3$ 
the critical value is 0. When $\lambda_c=0$, a simple argument implies that the transition is continuous.

 \cite{Chen-92, Chen-94} generalized Durrett and Gray's model on $\ZZ^2$. Let $e_1$, $e_2$ be the two unit vectors and define: 
\begin{center}
\begin{tabular}{cccc}
pair 1 & pair 2 & pair 3 & pair 4 \\
$x-e_1, x-e_2$ & $x+e_1, x-e_2$ & $x+e_1, x+e_2$ & $x-e_1, x+e_2$ 
\end{tabular}
\end{center}
Chen's models are numbered by the pairs that can give birth: 
Model I (pair 1 = SW corner rule), Model IIa (pairs 1 and 2), 
Model IIb (pairs 1 and 3), Model III (pairs 1--3), and Model IV (any pair).

\cite{Chen-92} proved for Model IV that if $0 < p < p(\lambda)$, then
$$
P\left( 0 \in \xi^p_t \right) \le t^{-c \log_{2\lambda}(1/p)}.
$$
If we add spontaneous births at rate $\beta$ and let $\xi^{0,\beta}_\infty$ 
be the limit as $t\to\infty$ for the system starting from all 0's (which exists by monotonicity), then for large $\lambda$,
$$
\lim_{\beta\to 0} P\left( 0 \in \xi^{0,\beta}_\infty \right) > 0.
$$
The second result says that the all 0's state is unstable to perturbation. When the limit is 0, the all 0's state is stable to perturbation. In this case, if $\lambda > \lambda_e$, then the perturbed system will have two stationary distributions for small $\beta$. This is how (DG3) was proved. \cite{Chen-94} shows that this is true for Model III.

 \cite{DuNe-94} considered a model with deaths at rate 1, and births at rate $\beta \binom{k}{2}/\binom{2d}{2}$ at vacant sites with $k$ occupied nearest neighbors. The mean field equation, which is derived by assuming that all sites are independent, in this case is
$$
\frac{du}{dt} = - u + \beta (1-u)u^2
$$
This ODE has $\beta_e = 4$ and $\beta_f=\infty$, i.e., there is a nontrivial fixed point for $\beta\ge 4$ but 0 is always locally attracting. 
They showed that in the limit of fast stirring both critical values converged to 4.5. This threshold is the point where the PDE 
$$
\frac{\partial u}{\partial t} = u'' - u + \beta u (1-u)
$$ 
has traveling wave solutions $u(t,x)=w(x-ct)$ with $c>0$. The largest fixed point of the mean field differential equation is 2/3 at 4.5, but based on simulations they conjectured that the phase transition was continuous.

\cite{N-94} considered the contact process with sexual reproduction in $d=1$ with long-range interaction in continuous time. In her model, the spatial locations are $\epsilon\mathbb{Z}$ and $\xi_t^{\epsilon}:\epsilon\mathbb{Z}\to\{0,1\}$ has the following dynamics:

\begin{enumerate}[(i)]
 \item Particles die at rate $1$.
 \item A pair of adjacent particles at $x$ and $x+\epsilon$ produces an offspring with rate $\lambda$, which is sent to a location $y$ with probability $k_{\epsilon}(x-y)$. $k_{\epsilon}$ is the offspring distribution kernel, derived from an exponentially decaying, symmetric probability kernel $k$ on $\mathbb{R}$. 
 \item The birth at $y$ is suppressed if $y$ is occupied.
\end{enumerate}
For this process, she showed that:

\begin{theorem}[{\cite[Theorem 1]{N-94}}]
 In the limit as $\epsilon\to 0$, starting from product measure, the density of particles, $u$, evolves as a solution to the integro-differential equation
 \begin{align}\label{IntegroDiffN}
  \frac{\partial u}{\partial t} = -u + \lambda(1-u)(k*u^2)
 \end{align}
 In addition, \eqref{IntegroDiffN} admits traveling wave solutions, and there is a nondecreasing function $c_k:(0,\infty)\to\mathbb{R}\cup\{\pm\infty\}$ giving the wave speed corresponding to $\lambda$ and $k$.
\end{theorem}

\begin{theorem}[{\cite[Theorem 3]{N-94}}]
 If $c_k(\lambda)>0$, then for small enough $\epsilon$, there is a nontrivial stationary distribution. Additionally, there is a constant $\lambda^*$ above which the wave speed is indeed strictly positive.
\end{theorem}

Neuhauser's result plays an important part in our first proof, but the main motivation for this work came from research done by  
\cite*{Getal1,Getal2,Getal3,Getal4,Getal5}.  
They considered a modification of Model IV in which particles 
hop according to the simple exclusion process at rate $h$. Birth rates at a site are 1/4 times the number of adjacent pairs of occupied sites, while deaths occur at rate $p$. Having $h>0$ means that $p_f(h)>0$. When $h$ is small, $p_f(h) < p_e(h)$, in which case ``both the vacuum and active steady state are stable.'' When $p_e(h) > p_f(h)$, this model has a discontinuous phase transition.

They defined a speed $V(p,h,S)$ using simulation of the process in a strip with slope $S$ and argued that
\begin{align*}
p < p_e(h) & \quad \hbox{ if some $V(p,h,S)>0$}\\
p < p_f(h) & \quad \hbox{ if all $V(p,h,S)>0$}
\end{align*}
It is an interesting problem to define the speeds rigorously 
for a class of attractive finite range process and to prove the 
relationships in the last display. Here, we avoid that problem by 
taking a long range limit to get an integro-differential equation 
for which the existence of speeds can be proved using results of  \cite{W82}.

We consider the two-dimensional contact process with sexual reproduction in discrete time (the results generalize to all higher dimensions as well). 
Let $k:\mathbb{R}^2 \to [0,1]$ be a probability density that is invariant under reflections in either axis: i.e., $k(x_1,x_2)=k(-x_1,x_2)$ and $k(x_1,x_2) = k(x_1,-x_2)$ and has compact support. At time $n \in \{0,1,2,\ldots \}$ the state of a site $x$ on the lattice $\mathbb{Z}^2/L$ is given by $\xi_{n}^{L}(x)$, which can take on the values $1$ (occupied) or $0$ (vacant). Starting with an initial configuration of particles, $\xi_{0}^{L}$ on $Z^2/L$, the process evolves in the following manner:
\begin{enumerate}[(i)]
\item At time $n$, given the configuration at the previous time $n-1$, with probability $\beta$, a vacant site $x$ on the lattice will generate a random variable $U_x$ with density $k$, choose the site $y$ closest to $x+U_x$, and then choose one of its four nearest neighbors $z$ at random. If both of the chosen sites are occupied, $x$ will also become occupied. 
\item After all births have occurred, with probability $\eta$, each particle is killed, independently of the others.
\end{enumerate}
If we take the limit of the particle system as $L\to\infty$ and ignore technical details, the density of 1's at time $n$ should satisfy
\begin{align}\label{IDE}
u_{n+1} = (1 - \eta)\left[u_n + \beta (1- u_n)\left(k*u_{n}^{2}\right)\right].
\end{align}
In the situation in which $u_n(x)\equiv v_n$ 
\begin{align}\label{Iter}
  v_{n+1} = (1 - \eta)\left[v_n + \beta (1- v_n)v_{n}^{2}\right]. 
\end{align}
When $\beta > 4\eta/(1-\eta)$, there are three equilibrium solutions of \eqref{Iter}: $0$ and $\rho_u < 1/2 < \rho_s$ given by
$$
\frac{1}{2} \left( 1 \pm \sqrt{1-\frac{4\eta}{\beta(1-\eta)}} \right)
$$

Using results of  \cite{W82}, we can prove existence of wave speeds under the assumption $\beta > 4\eta/(1-\eta)$. That is, if $S^1$ is the circle of radius one in the plane, then there exists a function $c^*:S^1 \to \mathbb{R}\cup \{-\infty,\infty\}$, which gives the speed of propagation of plane waves solutions of \eqref{IDE}  in the direction $\theta$.

\begin{theorem}\label{PosSpeeds}
If $c^*(\theta)> 0$ for all $\theta\in S^1$, then for large $L$, there is a nontrivial stationary distribution.
\end{theorem}

The proof of this result is based on  \citep{N-94}. There are three ingredients:
\begin{enumerate}[(i)]
\item a hydrodynamic limit, which shows that as $L\to\infty$ the particle system converges to a solution of \eqref{IDE}, 
\item a convergence result of \cite{W82} for solutions of \eqref{IDE}, and 
\item a block construction.
\end{enumerate}
Our symmetry assumption implies that if the speed in some direction $(\theta_1,\theta_2)$ 
is positive then there are a total of four directions obtained by
reflection across the axes that also have positive speeds. From this, we see that if 
a large region of 0's develops then the process will be able to fill in this hole and hence there should be a stationary distribution.
This intuition can be made rigorous by using a comparison with a ``generic population process'' of  \cite{BG-91}.

\begin{theorem}\label{NegSpeed}
If there exists a $(\theta_1,\theta_2)\in S^1$ 
and $c^*(\theta) > 0$, then for large $L$, 
there is a nontrivial stationary distribution.
\end{theorem}

Our application is more complicated than their proof of Toom's eroder theorem. In their proof the bad regions
are areas that contain 0's, while in our proof the bad regions are areas where the density of 1's
is not large enough in some small box. In \citep{BG-91}, the state outside a bad region is always the same, all 1's.
In this paper, we must use functions from Weinberger's proof of the existence of wave speeds to control the decrease of the bad region.

By the same logic, if some speed is negative and the initial configuration is finite then it can be put inside a shrinking parallelogram.

\begin{conjecture}
If there is a $(\theta_1,\theta_2)\in S^1$ with 
and $c^*(\theta) < 0$, the system cannot survive starting from an initial configuration with only a finite number of particles.
\end{conjecture}

\noindent
In a number of situations, for example see   \citep[Chapter 7]{CDP}, block constructions
have been used to show that a process dies out. However, in that situation, one shows that a large enough dead region 
will expand. Since our process is good at filling in holes, that conclusion is false here, and one will need a much different
method to prove our conjecture.

\section{Hydrodynamic Limit}

In this section, we will show that as $L\to\infty$, the particle system converges to the solution of the integro-differential equation
\eqref{IDE}. The first thing to do is to explain what it means for a sequence of configurations $\xi_n: \ZZ^2/L \to \{0,1\}$
to converge to a function $u_n(x)$. To do this, we will let $\gamma \in(0,1)$ and partition the plane into squares with
side $L^{-\gamma}$ whose corners are at the points $L^{-\gamma}(m,n)$ with $m,n \in \ZZ$.  
For any $x \in \mathbb{Z}^2/L$, let $x^*$ be the bottom-left corner of the box containing $x$, 
and let 
\begin{equation}\label{eq:BL}
B^{L}(x) = x^* + \big[0, L^{-\gamma})^2.
\end{equation} Each such box contains 
$\sim L^{2-2\gamma}$ points.
 
Let $S_{n}^{L}(x)$ be the number of particles in $B^{L}(x)$ at time $n$:
\begin{equation*}
 S_{n}^{L}(x) = \sum_{y\in B^{L}(x)}\xi_{n}^{L}(y).
\end{equation*}
Define $\xi_{n}^{L} \sim u_n$ to mean that for all $K$, as $L\to\infty$
\begin{align*}
 \sup_{x \in [-K,K]^2  \cap  \mathbb{Z}^2/L}  \left\vert \frac{S_{n}^{L}(x)}{L^{2-2\gamma}} - u_{n}(x^*) \right\vert \to 0
\end{align*}
Let $R_{n}^{L}(x) = $ number of pairs of adjacent particles in $B^L(x)$ at time $n$:
$$
 R_{n}^{L}(x) = \sum_{y\in B^{L}(x)} \zeta_n^L(y) \quad\hbox{where}\quad
\zeta_n^L(y)  = \xi_{n}^{L}(y) \cdot \frac{1}{2} [\xi_n^L(y+e_1) + \xi_n^L(y+e_2)],
$$
where $e_1$ and $e_2$ are unit vectors of $\mathbb{Z}^2/L.$ 
Define $\zeta_{n}^{L} \sim u_n^2$ to mean that for all $K<\infty$, as $L \to \infty$
\begin{align*}
 \sup_{x \in [-K,K]^2  \cap  \mathbb{Z}^2/L}  \left\vert \frac{R_{n}^{L}(x)}{L^{2-2\gamma}} - u_{n}(x^*)^2 \right\vert \to 0
\end{align*}

\begin{theorem}\label{HydrodLimit}
Suppose that $u_0(x):\mathbb{R}^2\to [0,1]$ is continuous
and that the sequence of initial configurations $\xi_{0}^{L} \sim u_0$ and $\zeta_{0}^{L} \sim u_{0}^{2}$. 
Then as $L \to \infty$, $\xi_{n}^{L} \sim u_n$ and $\zeta^L_n \sim u_n^2$, where $u_n$ is the solution of
$$
u_{n+1} = (1 - \eta)\left[u_n + \beta (1- u_n)\left(k*u_{n}^{2}\right)\right].
$$
\end{theorem}

\begin{proof}
By induction, it suffices to prove the result for $n=1$. Let $T_{x}$ be the translation by $x$.
In order to simplify the computation of expectations and variances, we will modify the birth step so
that the first parent $y$ will be chosen according to the probability kernel $T_{x^*}k_{L}$, instead of $T_{x}k_{L}$. 
As before, the second parent is chosen at random, with equal probability, from the nearest neighbors of the first parent.
The next lemma shows that at time $1$ the two processes are equal with high probability.

\begin{lemma}\label{Claim}
If $\xi_{0}^{L}(x) = \hat{\xi}_{0}^{L}(x) $ for each $x\in \mathbb{Z}^2/L$, then
$$
P\left(\xi_{1}^{L}(x) \neq \hat{\xi}_{1}^{L}(x)\right) \to 0 \text{  as  } L \to \infty.
$$
\end{lemma}

\begin{proof}
For $y\in\mathbb{Z}^2/L$, let $\alpha_L(y) = k_L(y-x)$, $\beta_L(y) = k_L(y-x^*)$, and let
$$
p_L = \sum_{y\in \ZZ^2/L} \min\{\alpha_L(y), \beta_L(y)\}
$$
A standard argument from analysis shows that if $f\in L^1(\RR^2)$ and $\delta =(\delta_1,\delta_2) \to (0,0)$, 
then $\| T_\delta f - f \|_1 \to 0$ (approximate $f$ by a continuous function $g$). 
This implies that as $L\to\infty$, $p_L \to 1$ uniformly for $x\in \mathbb{Z}^2/L$.

To couple the two processes, use the same coin flips to see if births should occur.
If a birth event occurs at site $x\in Z^2/L$ at time $1$, flip a coin with probability $p_L$ of heads. If heads comes up, then with probability 
$$\min\{\alpha_L(y),\beta_L(y)\}/p_L$$ choose $y$ as the first parent for $x$ in both processes, and then make the same choice of second parent $z$. 
Otherwise, the particle in $\xi_L$ chooses  its first parent $y$ with probability $$(\alpha_L(y)-\beta_L(y))^{+}/(1-p_L),$$ 
the particle in $\hat{\xi}_L$ chooses its first parent $y$ with probability $$(\beta_L(y)-\alpha_L(y))^{+}/(1-p_L),$$ and the choice
of second parents are made independently. Once the births have been done, we use the same coin flips to decide the deaths and the proof is complete.
\end{proof}

Let $P_{0}$ denote the probability law for the process $\hat{\xi}_{n}^L$ with initial configuration $\hat{\xi}_{0}^{L}$. In order to write out the computations more compactly, introduce the notations 
$X_i = \hat{\xi}_{i}^{L}(x)$, $Y_i = \hat{\xi}_{i}^{L}(y)$, and $ Z_i = \hat{\xi}_{i}^{L}(z)$
for $i=0,1$ and $x,y,z\in\mathbb{Z}^2/L$, and set 
\begin{equation*}
K_L(x) = \sum_{y\in\mathbb{Z}^2/L} k_{L}(y-x^{*})\hat{\xi}_{0}^{L}(y) \cdot \frac{1}{4} \sum_{z\sim y} \hat\xi^L_0(z) .
\end{equation*}
so that $$P_{0} \left(X_1 = 1 \right) = (1 - \eta)\left[X_0 + \beta\left( 1 - X_0 \right) \cdot K_L(x) \right] := p_x.$$
Since $X_1$ is Bernoulli, $\text{Var}_{0}\left(X_1\right) = p_x - p_{x}^{2}$.

The total number of particles alive in $B^{L}(x)$ at time $1$ is  
$$
\hat{S}_{1}^{L}(x) = \sum_{y\in B^{L}(x)}\hat{\xi}_{1}^{L}(y),
$$
so the expected proportion of occupied sites in $B^{L}(x)$ at time $1$ is 
\beq\label{ExpHat}
\mathbb{E}_{0}\left( \frac{\hat{S}_{1}^{L}(x)}{L^{2-2\gamma}} \right) 
 = (1 - \eta)\left[\frac{\hat{S}_{0}^{L}(x)}{L^{2-2\gamma}} + \beta \left(1 -  \frac{\hat{S}_{0}^{L}(x)}{L^{2-2\gamma}}\right)K_L(x)\right].
\eeq
For $y \neq z$ in $B^{L}(x)$, 
\begin{align*}
\text{Cov}_{0} &\left( Y_1,Z_1 \right)  =  \mathbb{E}_{0} \left(Y_{1}Z_1\right) - \mathbb{E}_{0}\left(Y_1\right)\mathbb{E}_{0}\left(Z_1\right) \\
     & =  (1-\eta)^2 \bigg[ Y_0Z_0 + \beta\cdot K_L(x)\left(\big(1-Y_0\big)Z_0 + Y_0\big(1-Z_0\big)\right) +\: \left(\beta\cdot K_L(x)\right)^2\big(1-Y_0\big)\big(1-Z_0\big)\bigg] - p_{y}p_{z} = 0. 
\end{align*}
It is for this reason we modified our process so that all sites in $B_L(x)$ use the kernel $k_l(\cdot - x^*)$.

Since the covariance is 0, 
$$
\text{Var}_{0}\left( \frac{\hat{S}_{1}^{L}(x)}{L^{2-2\gamma}}\right) = \frac{1}{L^{4-4\gamma}}\sum_{y\in B^{L}(x)}\left(p_{y} - p_{y}^{2}\right)
\leq \frac{1}{L^{4-4\gamma}}\sum_{y\in B^{L}(x)}p_{y}
$$
The variance of $\hat{S}_{1}^{L}(x)/L^{2-2\gamma}$ is
\begin{align*}
   \le  \frac{1}{L^{4-4\gamma}}\left[ \hat{S}_{0}^{L}(x) + \beta \left(L^{2-2\gamma} -  \hat{S}_{0}^{L}(x)\right)\right]  
\leq  C\cdot \frac{1}{L^{2-2\gamma}}.
\end{align*}
where $C$ does not depend on $L$. By Chebyshev's inequality, 
\begin{align*}
P_{0}\left( \left\vert \frac{\hat{S}_{1}^{L}(x)}{L^{2-2\gamma}} 
- \mathbb{E}_0\left(\frac{\hat{S}_{1}^{L}(x)}{L^{2-2\gamma}}\right) \right\vert \geq \delta \right)
 \leq \delta^{-2} \text{Var}\left( \frac{\hat{S}_{1}^{L}(x)}{L^{2-2\gamma}} \right) \leq \frac{C}{\delta^{2}L^{2-2\gamma}}.
\end{align*}
There are $4L^{2\gamma}\cdot K^2$ boxes of side length $L^{-\gamma}$ in each $[-K,K]^2$ box, so 
\begin{align}\label{ChebHat}
P_{0}\left( \sup_{x \in [-K,K]^2  \cap  \mathbb{Z}^2/L}  \left\vert \frac{\hat{S}_{1}^{L}(x)}{L^{2-2\gamma}} - \mathbb{E}_0\left(\frac{\hat{S}_{1}^{L}(x)}{L^{2-2\gamma}}\right) \right\vert \geq \delta \right) \leq \frac{CL^{2\gamma}}{\delta^{2}L^{2-2\gamma}}. 
\end{align}
For $\gamma \in (0,\frac{1}{2})$, this probability tends to $0$ as $L \to \infty$.
For $x\in\mathbb{R}^2$ and $x_L\in \mathbb{Z}^2/L$ such that $x_L\to x$  as $L\to\infty$, $$K_L(x_L)\to (k* u_0^2)(x).$$ By the assumptions of Theorem \ref{HydrodLimit}, $$\hat{S}_{0}^{L}(x)/L^{2-2\gamma}= S_{0}^{L}(x)/L^{2-2\gamma} \to u_{0}^{}(x).$$ Together with \eqref{ExpHat}, this implies that
\begin{align*}
 \mathbb{E}_0\left(\frac{\hat{S}_{1}^{L}(x_L)}{L^{2-2\gamma}}\right) \to u_1(x) \quad \text{ as } \quad L\to\infty.
\end{align*}
Then by \eqref{ChebHat}, when $L\to\infty$,  $$\hat{S}_{1}^{L}(x_L)/L^{2-2\gamma} \to u_1(x)$$ and thus $\hat{\xi}_{1}^{L} \sim u_1$, and from Lemma \ref{Claim} 
it follows that $\xi_{1}^{L} \sim u_1$.

Our next step is to prove a result for pairs. Let
$$
\hat R_{n}^{L}(x) = \sum_{y\in B_*^{L}(x)} \hat\zeta_n^L(y) \quad\hbox{where}\quad
\hat\zeta_n^L(y)  = \hat\xi_{n}^{L}(y) \cdot \frac{1}{2} \left[\hat\xi_n^L(y+e_1) + \hat\xi_n^L(y+e_2)\right]
$$
and $B_*^{L}(x)$ is the set of points $y$ so that $y+e_1$ and $y+e_2$ are also in the box.
As argued in the first part of the proof, given the initial configuration the $\hat \xi^L_1(y)$ are independent
so when $y,z \in B^L(x)$ we have $$E\left(\hat\xi_{n}^{L}(y) \hat\xi_n^L(z)\right) = p_x^2.$$ The pairs $\left(\hat\xi^L_1(y), \hat\xi^L_1(z)\right)$ for $y \in B^L_*(x)$ and $z = y + e_1$ or $y+e_2$ are not independent when $\{y,z\} \cap \{y',z'\} \neq\varnothing$ but one still has the estimate
$$
\text{Var}_{0}\left( \frac{\hat{R}_{1}^{L}(x)}{L^{2-2\gamma}} \right) \le \frac{C}{L^{2-2\gamma}}
$$
so the previous proof can be repeated almost word for word to conclude that $\hat \zeta^L_1 \sim u_1^2$ and by Lemma \ref{Claim} 
$\zeta^L_1 \sim u_1^2$.
\end{proof} 

With the hydrodynamic limit established, the proof of Theorem \ref{PosSpeeds} is routine. 
Let $\delta > 0$ be small enough so that $\rho_s - 2 \delta > \rho_u + 2\delta$.  \citep[Theorem 6.2]{W82} implies that

\begin{lemma} \label{Wlemma}
If $K$ is large enough and $N \ge N_K$ then $u_0(x) > \rho_u + \delta$ on $[-K,K]^2$ implies that $u_N(x) > \rho_s - \delta$ on $[-4K,4K]\times [-K,K]$.
\end{lemma}

Let $$I_n = 2nK + [-K,K]^2.$$ We say that $I_n$ is good at time $t$ if we have $S^L_t(x) \ge \rho_u + 2\delta$ for all $x\in I_n$. Combining Theorem \ref{HydrodLimit}
with Lemma \ref{Wlemma} gives

\begin{lemma} \label{induct} 
Let $\ep>0$. If $L$ is large enough and $I_0$ is good at time 0, then with probability $\ge 1-\ep$, $I_1$ and $I_{-1}$ are good at time $N$.
\end{lemma}

\noindent
Since the process cannot move by more than distance 1 in one time step, the events involved in the last lemma have a finite range of
interaction than only depends on $N$. It follows from this and results in \citep{D-95} that if $\ep < \ep_N$, the
good sites dominate a supercritical oriented percolation and hence there is a nontrivial stationary distribution.

\section{Proof of Theorem \ref{NegSpeed}}

The result will be proved using the comparison model of \cite{BG-91}. 
The goal is to show that our quadratic contact process and their comparison process 
can be coupled so that regions in the quadratic contact process with a low density 
of particles are entirely contained in the vacant region of the comparison model. 
Then the existence of a nontrivial stationary distribution for the comparison model 
will imply the existence of a nontrivial stationary distribution for the quadratic contact model.

In the comparison process of Bramson and Gray, at each time $t>0$, the 
plane $\mathbb{R}^2$ is divided into two regions: a vacant and a nonvacant 
region. Let $A_t$ denote the vacant region of the comparison process at time 
$t$ and take $A_0 = \varnothing$. Fix unit {\it orientation vectors} 
$n_1, n_2, n_3 \in S^1$. $A_t$ will be a union of triangles, each 
triangle having edges perpendicular to the orientation vectors. For there to exist triangles 
with edges perpendicular to the $n_i$, it must be the case that the $n_i$, $i=1,2,3$ span 
$\mathbb{R}^2$ and that 
\begin{equation}\label{directionCondition}
n_1 = -\alpha_2 n_2 - \alpha_3 n_3,
\end{equation}
for some $\alpha_2,\alpha_3 > 0$.
The edges of the triangles move inward at linear rates, to be specified below. 
$T(x,s;t)\subset A_t$ will denote the triangle that first appears at time $s$ and 
with incenter at $x\in\mathbb{R}^2$.

New triangles that are part of the vacant region are created in two ways:

I. Let
$\mathcal{P}$ be a Poisson process on $\mathbb{R}^2 \times [0,\infty )$ with intensity $\epsilon$. 
If $(x,s) \in \mathcal{P}$, a triangular vacant region with edges perpendicular to the orientation 
vectors and with incenter at $x \in \mathbb{R}^2$ and having radius $r$, is created at time $s$. 
The $i^{th}$ edge of this triangle, corresponding to the orientation vector $n_i$  will move inward with 
linear rate $a_i$, $i=1,2,3$ (in general $a_i's$ may be negative, and in this case the 
corresponding edges move outward). The collection of all such triangles at time $t$
is what Bramson and Gray call the {\it noninteractive region} at $t$ and the individual triangles, they call
{\it death regions}.

II. The second way that new triangles can appear in the vacant region takes into account 
possible overlaps and collisions of the shrinking and/or expanding triangles that make up 
the vacant region. If $(x,s)\in\mathcal{P}$, and there exist $T_k(x_k,s_k;s)\in A_s$ such that 
$$T(x,s;s)\cap \left(\bigcap_{k} T_k(x_k,s_k;s)\right) \neq \varnothing,$$ 
taking the maximal such intersection, this intersection forms an 
{\it overlap region} that is created at time $s$.
This region will also be triangular with edges perpendicular to $n_i$, $i=1,2,3$ \citep[Proposition~1]{BG-91}.
A single death region $T(x,s;s)$ can give rise to more than one overlap region.
The edges of the overlap region move outward at positive rate $b$, the 
{\it interaction rate}, until each edge catches up to the corresponding edges 
of all of the regions that produced the original overlap. At that 
point, the edges of the overlap or collision region will begin to move inward 
with rates $a_i$. In addition to overlaps from newly formed death regions
that overlap already existing regions, triangles with outward moving edges 
can collide with each other: if there exist $T_k(x_k,s_k;s)\in A_s$ such that 
$$T(x,s;s)\cap \left(\bigcap_{k} T_k(x_k,s_k;t)\right) = \varnothing,$$ 
for $t < s$, but 
$$T(x,s;s)\cap \left(\bigcap_{k} T_k(x_k,s_k;s)\right) = \{x\}.$$
In this case, a collision region is formed, starting out as a triangle with inradius $0$ at the single point $\{x\}$,
and with edges perpendicular to the $n_i$, $i=1,2,3$, moving outward at rate $b$ until they have caught up with 
all of the corresponding edges of the triangles that initiated the collision.

   If $a_i > 0$, for 
$i=1,2,3$, the process has a nontrivial 
stationary distribution when the error rate $\epsilon$ is small enough \citep[Theorem~1]{BG-91}.

Coming back to the quadratic contact process, take $\xi_1$ to be a direction with 
positive wave speed $\alpha_1 > 0$, since we have assumed that there is at least one positive 
speed. Because the offspring distribution kernel has $\mathbb{Z}^2$-symmetry, and because the 
wave speed $c^*(\xi)$ is a lower semi-continuous function in $\xi$ \cite[Proposition~5.1]{W82}, 
the existence of a positive speed in one direction implies the existence of a positive speed 
in three directions $\xi_1, \xi_2,$ and $\xi_3$, that span $\mathbb{R}^2$ and satisfy \eqref{directionCondition}.
Let $\alpha_2$ and $\alpha_3$ be the wave speeds corresponding to $\xi_2$ and $\xi_3$, respectively.

We shall say that the box $B^{L}(x)$ has {\it low density} at time $n$ 
(see~\eqref{eq:BL} for the definition of $B^L(x)$) if the density 
of the particles inside $B^{L}(x)$ falls below an appropriate 
threshold $\alpha \in (\rho_u,\rho_s)$, to be specified below. 
We will refer to boxes that do not have low density as good boxes.

Recall that $0 < \rho_u < \rho_s < 1$ are the nonzero fixed points of the operator
\begin{align*}
 Q[u] = \left(1 - \eta\right) \left[u + \beta\left(1 - u\right)\left(k*u^{2}\right) \right],
\end{align*}
 which exist when $\beta > 4\eta/(1-\eta)$.
 
 For what follows, we need that
 
 \begin{lemma}\label{Monotone}
  $Q$ is monotone, i.e. if $u\leq v$, then $Q[u]\leq Q[v];$ furthermore, if $u < v$, then $Q[u]< Q[v].$

 \end{lemma}
 \begin{proof}
  \begin{align*}
   Q[v]-Q[u]  & = \left(1 - \eta\right)\left[\left(v - u\right) + \beta\left((1-v)k\ast v^2 - (1-u)k\ast u^2\right)\right]\\
  & \geq   \left(1 - \eta\right)\left[\left(v - u\right) + \beta\left((1-v)k\ast v^2 - (1-u)k\ast v^2\right)\right] \\
   &=  \left(1 - \eta\right)\left(v - u\right)\left(1 - \beta \left(k\ast v^2\right)\right)\\
   &\geq 0.
  \end{align*}
  The first inequality holds since $0 \leq u \leq v$, and the second 
  inequality holds because $v\geq u$ and $0\leq k\ast v^2 \leq 1$ 
  since $k$ is a probability kernel and $0\leq v \leq 1$. If $u < v$, the 
  last inequality is strict.
 \end{proof}

The collection of low density boxes at time $n$ will be contained in the vacant 
region of the comparison process. Let $\tilde{A}_n$ be the collection of all
low density boxes of the contact process on $\mathbb{Z}^2/L$ 
at time $n$. The process initially has all sites occupied, so that $\tilde{A}_0 = \varnothing$. 

Let $\psi: \mathbb{R} \to \mathbb{R}$ be a function with the following 
properties (see Figure~\ref{fig:1} for an example):
\begin{enumerate}[i)]
\item $\psi$ is continuous,
\item $\psi$ is nonincreasing,
\item $\psi(-\infty) \in (\rho_u,\rho_s)$ and $\psi(s) = 0$ for $s \geq 0$.
\end{enumerate}
(We consider $s\in\mathbb{R}$ as just a real variable here, with no connection to the time $s\in[0,\infty)$ in the previous section.)

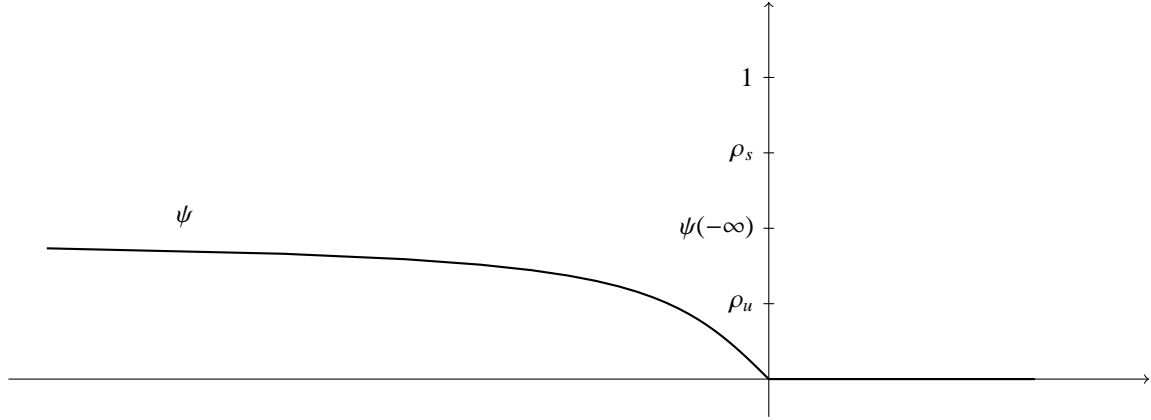
\begin{figure}
 \begin{tikzpicture}[domain=-9:0] 
    \draw[->] (-10,0) -- (5,0) node[right] {}; 
    \draw[->] (0,-0.5) -- (0,5) node[above] {};
    \draw[color=black, thick,rotate=90,scale=1.2]   plot  [domain=0:46*pi/100] (\x,{tan(\x r)})  (1.8cm,6.2cm)  node[left] {$\psi$};
     \draw[color=black, thick]   plot  [domain=0:3.5] (\x,0);  
   \draw (2pt,1 cm) -- (-2pt,1 cm) node[anchor=east,fill=white] {$\rho_u$};
 \draw (2pt,2 cm) -- (-2pt,2 cm) node[anchor=east,fill=white] {$\psi(-\infty)$};
  \draw (2pt,3 cm) -- (-2pt,3 cm) node[anchor=east,fill=white] {$\rho_s$};
   \draw (2pt,4 cm) -- (-2pt,4 cm) node[anchor=east,fill=white] {$1$};
  \end{tikzpicture}
     \caption[]{$\psi$ must be continuous, nonincreasing, and $\psi(-\infty\in (\rho_u,\rho_s)$. 
     With these conditions satisfied, $c^*(\xi)$ is independent of the choice of $\psi$ \citep[Section~5]{W82}.}
     \label{fig:1}
\end{figure} 

As in \cite[Section~5]{W82}, for $c \in \mathbb{R}$ and $\xi \in S^1$, the unit circle in $\mathbb{R}^2$,
we consider the operator on continuous functions from 
$\mathbb{R}$ to $\mathbb{R}$
$$
R_{c,\xi}[f](s) := \max\{\psi(s),Q[f(x\cdot\xi+s+c)](0,0)\},\quad s\in\mathbb{R},
$$
where $c, s,$ and $\xi$ are held fixed and $f$ is considered as a function of $x\in\mathbb{R}^2$.
Let $f_0 = \psi$, and $$f_{n+1} = R_{c,\xi}[f_n],\quad n=0,1,\ldots.$$
By \cite[Lemma~5.1]{W82}, for each $s \in\mathbb{R}$, the sequence $\{f_n(s)\}$ is 
nondecreasing in $n$, and, for each $n=0,1,\ldots$, the function $f_n$ is nonincreasing 
in its argument. Moreover, $\{f_n\}$ increases to a limiting function $f$ 
with $f(-\infty) = \rho_s$.  If 
$c<c^*$, where $c^*$ is the wave speed in the direction $\xi$, then $f(\infty) = \rho_s$. 
Indeed, $c^*$ is independent of the initial choice of $\psi$ satisfying the above conditions.  
Thus, if $c<c^*$, $$f\equiv \rho_s$$
(see \cite[Section~5]{W82}). Let
\begin{equation*}\label{eq:defc}
c \in (0, \min\{\alpha_i/2: i=1,2,3 \}),
\end{equation*}
and note that $c>0$.
For each of the directions $\xi_i$ with positive wave speed $\alpha_i$, $i=1,2,3$, 
consider the sequence $\{f_{n,i}\}$ where $f_{0,i} = \psi$ 
and $$f_{n+1,i} = R_{c,\xi_i}[f_{n,i}],\quad n=0,1,\ldots.$$

\begin{figure}[b]
 \begin{tikzpicture}[domain=-9:0] 
    \draw[->] (-10,0) -- (5,0) node[right] {}; 
    \draw[->] (0,-0.5) -- (0,5) node[above] {};
       \draw (2pt,0.925 cm) -- (-2pt,0.925 cm) node[anchor=east,fill=white] {$\rho_u$};
 \draw (2pt,1.85 cm) -- (-2pt,1.85 cm) node[anchor=east,fill=white] {};
  \draw (2pt,2.775 cm) -- (-2pt,2.775 cm) node[anchor=east,fill=white] {$\rho_s$};
   \draw (2pt,3.7 cm) -- (-2pt,3.7 cm) node[anchor=east,fill=white] {$1$};
     \draw (3.78,3pt) -- (3.78,-3 pt) node[below] {$M$};
          \draw (1,3pt) -- (1,-3 pt) node[below] {$m$};
    \draw[color=black, thick,dashed,rotate=90,scale=1.2]   plot  [domain=0:46*pi/100] (\x,{1.3*tan(\x r)-2})  (1.1cm,6.8cm)  node[left] {$\phi$};
        \draw[color=black, thick,rotate=90,scale=1.9]   plot  [domain=0:45*pi/100] (\x,{tan(\x r)-2})  (0.85cm,-1.7cm)  node[left] {$Q[\phi]$};
     \draw[color=black, thick,dashed]   plot  [domain=2.4:3.78] (\x,0);  
         \draw[color=black, thick]   plot  [domain=3.78:4.5] (\x,0);  
          \draw[color=black,dashed]   plot  [domain=-9.5:1] (\x,1.85);  
  \end{tikzpicture}
\end{figure}

For the coupling with the comparison process, the following will play an important role

\begin{lemma}
 There exists $n_0\in\mathbb{N}$ such that for all $n\geq n_0$ and for each $i = 1,2,3$,
 \begin{equation}\label{Translates}
  f_{n,i}(s-c) \leq Q\left[f_{n,i}\left(x\cdot\xi_i + s\right)\right](0,0)
 \end{equation}
\end{lemma}

\begin{proof}
 By definition,
 $$f_{n,i}(s-c) = R\left[f_{n-1,i}(s-c)\right] = max\left\{\psi(s-c),Q\left[f_{n-1,i}\left(x\cdot\xi_i +s\right)\right](0,0)\right\}$$
 Since $c<c^*_i$, by definition of $c^*_i$ (see \cite[Section~5]{W82}), for each $s\in\mathbb{R}$, 
 $$\lim_{n\to\infty}f_{n,i}(s) = \rho_s.$$
since $\psi$ is nonincreasing in $s$ and $\psi(-\infty) < \rho_s$, for some large enough $n_i\in\mathbb{N}$, 
we must have that $$\psi(s-c) \leq Q\left[f_{n_i-1,i}\left(x\cdot\xi_i +s\right)\right](0,0).$$
Thus for $n\geq n_i$,
$$f_{n,i}(s-c) = Q\left[f_{n-1,i}\left(x\cdot\xi_i +s\right)\right](0,0)\leq Q\left[f_{n,i}\left(x\cdot\xi_i +s\right)\right](0,0),$$
the inequality due to the fact that $f_{n,i}$ is nondecreasing in $n$ and $Q$ is monotone. We take $n_0 = \max_{i=1,2,3}n_i$.
 \end{proof}
\begin{lemma}
There exists $n_0\in\mathbb{N}$ and $c>0$ such that for all $n\geq n_0$ and for each $i = 1,2,3$, and 
there exists  $\psi : \mathbb{R}\to\mathbb{R}$, continuous, nonincreasing, with $\psi(-\infty) \in (\rho_u,\rho_s)$ 
and $\psi(s) = 0$ for $s \geq 0$, such that
\begin{equation}\label{TranslatesStrict}
  f_{n,i}(s-c) < Q\left[f_{n,i}\left(x\cdot\xi_i + s\right)\right](0,0)
 \end{equation}
 for $s\in\mathbb{R}$ with $f_{n,i}(s-c) > 0.$
\end{lemma}
\begin{proof}
Take $\psi : \mathbb{R}\to\mathbb{R}$ to be a continuous, strictly decreasing function with $\psi(-\infty) \in (\rho_u,\rho_s)$ 
and $\psi(s) = 0$ for $s \geq 0$. By induction, we will show that $f_{k,i}(s) = R\left[f_{k-1,i}(s)\right]$, with $f_{0,i}(s) = \psi(s)$ 
is also strictly decreasing for $s$ for which $f_{k,i}(s) > 0$  and for each $k=0,1,2,\ldots$. Suppose that $f_{k,i}(s)$ is strictly decreasing for 
$s$ for which $f_{k,i}(s) > 0$.
$$f_{k+1,i}(s) =  max\left\{\psi(s),Q\left[f_{k,i}\left(x\cdot\xi_i +s + c\right)\right](0,0)\right\}, $$
and $Q\left[f_{k,i}\left(x\cdot\xi_i +s + c\right)\right](0,0)$ is strictly decreasing in $s$ by \eqref{Monotone} since $f_{k,i}$ is 
strictly decreasing in $s$. Since  $f_{k+1,i}$ is the maximum of two strictly decreasing functions, it is also strictly decreasing.
Since \eqref{Translates} holds for any $c\in(0,\min\{\alpha_i/2: i=1,2,3 \})$ and both functions are strictly decreasing in $s$ as long as they are non-zero, 
for $c = \min\{\alpha_i/2: i=1,2,3 \}$, strict inequality must hold.
\end{proof}

Suppose $f_{n_0,i}$ satisfies \eqref{TranslatesStrict} for $s\in\mathbb{R}$ with $f_{n_0,i}(s-c) > 0.$
Now, set $\alpha = \min_i\left\{f_{n_0,i}(-\infty)\right\}$ and let
\begin{align*}
m = \min_{i}\sup\big\{s: Q[f_{n_0,i}(x\cdot\xi_i+s)](0,0) = \alpha\big\}
& & \text{ and } & & M = \max_i\inf\big\{s: Q[f_{n_0,i}(x\cdot\xi_i +s)](0,0) = 0 \big\}.
\end{align*}
and set $l := M - m$.

The coupling will take place in the following way.
If an {\it error} occurs in the contact process (the detailed description of what constitutes an error is below) 
such as a box $B^L(x)$ that is good at time $n-1$ becomes a low density box at time $n$, 
a triangular vacant region $\tilde{T}(x,n;n)$ will appear in the contact process, and
in the comparison 
process $B^L(x)$ is also covered by a corresponding triangular region, $T(x_*,n_*;n)$, $n_*\in[n,n+1)$ chosen 
uniformly at random, 
and $x_*\in\mathbb{R}^2$ chosen uniformly at random from $B^L(x)\subset\mathbb{R}^2$,
of an appropriate size, so that the center of the box is at the center of the triangle in the contact process
with inscribed circle of radius $$r = \lceil l + 2d(B) + 2c + d(k)\rceil,$$ initially 
at time $n$, where $d(B) = L^{-\gamma}\sqrt{2}$ is the diameter of the box $B_L(x)$ and 
$$d(k) = \sup \left\{|x-y|: x,y\in\mathbb{R}^2, k(x)\neq 0, k(y)\neq 0 \right\}$$ is 
the diameter of the kernel. 
In the comparison process, the corresponding triangle has the same center but inscribed 
circle of radius $$r_* := r-(n+1 - n_*).$$
Once such a region is created, the edges of $T(x_*,n_*;n)$ in the comparison process move inward, 
each with linear rate $c$. While the edges of $\tilde{T}(x,n;n)$ will have moved inward by $c$ units 
from time $n$ to $n+1$, since the contact process is in discrete time. The region 
in the comparison process
at time $t$, $t\geq n_*$ will be denoted by $T(x_*,n_*;t)$.

Set the \textit{interaction rate} $b$ to be 
$r$. In the comparison process, 
should two or more triangles collide or overlap, a new {\it collision} or {\it overlap} 
region is formed at that time, which is the intersection of the maximal collection of 
the colliding or overlapping regions that has nonempty intersection. The edges of the 
new collision or overlap region then move outward at rate $b$ in each direction until 
they have caught up to all of the respective edges of the regions that initiated the 
overlap or collision in each respective direction, and after that the edge of the 
overlap or collision region will again move inward with rate $c$. By construction, for integer times $m>n$, $\tilde{T}(x,n;m) = T(x_*,n_*;m)$.

\tikzstyle{buffer1} = [draw,shape border rotate=90, thick, red, dashed,isosceles triangle, node distance=1.7cm, minimum height=2.6em] 
\tikzstyle{buffer2} = [draw,shape border rotate=90, thick,isosceles triangle, node distance=2cm, minimum height=4em] 
\tikzstyle{buffer3} = [draw,shape border rotate=90, thick, blue, dashed,isosceles triangle, node distance=7cm, minimum height=10em]
\tikzstyle{buffer4} = [draw,shape border rotate=90, thick, red, dashed,isosceles triangle, node distance=2.9cm, minimum height=5.1em] 
\tikzstyle{buffer5} = [draw,shape border rotate=90, thick, red, dashed,isosceles triangle, node distance=2.8cm, minimum height=5em] 
\tikzstyle{buffer6} = [draw,shape border rotate=90, thick, isosceles triangle, node distance=3.5cm, minimum height=6.5em] 
\tikzstyle{buffer7} = [draw,shape border rotate=90, thick, isosceles triangle, node distance=3.4cm, minimum height=6.4em]  
  \begin{figure}[b]
 \begin{tikzpicture}
 \draw (0 cm,0 cm) node[buffer1]{}  node[buffer2]{};
 \draw (5cm,-2cm) node  [buffer3]{};
  \draw (5.83cm,-2.55cm) node [buffer6]{};
 \draw (4.6cm,-1.5cm) node [buffer7]{};
 \draw (5.85cm,-2.57cm) node [buffer4]{};
  \draw (4.6cm,-1.5cm) node [buffer5]{};
  \node at (9.07,0) {$t_0$};
  \draw (9.3,0) -- (9.9,0);
  \node at (9,-0.75) {$t$};
  \draw [dashed](9.3,-0.75) -- (9.9,-0.75);
  \end{tikzpicture}
 \caption{An overlap region is formed at some time $t_0$. 
 The edges of the overlap region move outward, and by time $t>t_0$, 
 they have caught up with the edges of the triangles initiating the overlap, 
 so all edges are now moving in. The overlap region at time $t$ is in blue. 
 The vacant triangular regions at time $t$ have become smaller and are in red.}
 \end{figure}
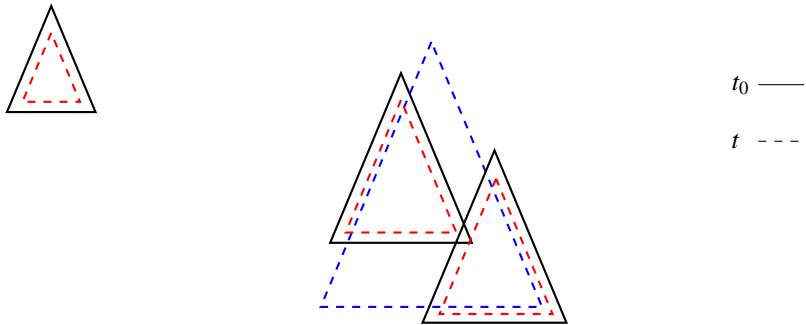

Bramson and Gray use a Poisson point process in their comparison model. 
We will consider a point process of errors, $\mathcal{P}$, 
derived from the quadratic contact process, which will be described below.
 Although $\mathcal{P}$ is not quite a Poisson point proccess, 
it shares two key properties, \eqref{P1} and \eqref{P2}, with a Poisson point process.
In fact, the results of Bramson and Gray still hold
with such an underlying point process as $\mathcal{P}$, satisfying \eqref{P1} and \eqref{P2},
  instead of the Poisson process.

\subsection{Type I and type II errors}
 
To describe $\mathcal{P}$, we will consider  
two different types of errors that are possible, type I and type II errors. 
Define
\begin{equation}\label{translates}
 h_i(s) = 
 \begin{cases}
\alpha & \text{for } s \leq -(r-d(k)-c) \\
f_{n_0,i}(s) & \text{for } s > -(r-d(k)-c)
\end{cases}
\end{equation}

When either error occurs at some $x\in\mathbb{R}^2$, 
a triangular region $\tilde{T^e}(x,n;n)$ is created with 
inscribed circle of radius $r$. The edges of $\tilde{T^e}(x,n;n)$ 
move inward with linear speed $c$ until the triangle closes and no longer exists.
 If $B(y; d(k))\notin \tilde{A}_{n}$, for each $k=0,1,\ldots,m$, where $B(y;r)$ represents the 
 closed ball of radius $r$ centered at $y$,
the expected density 
of $B^L(y)$ is greater than $Q^{k}(\alpha) > \alpha$, where $Q^k$ represents $k$ iterations 
of $Q$.
As long as no other error regions of $\tilde{A}_{n+k}$ overlap $\tilde{T^e}(x,n;n +k)$,
due to the low density box 
at $x$, the expected density of boxes inside $\tilde{T^e}(x,n;n +m)$ 
 is at least greater than 
$h_i(\xi_i\cdot(x-y) - mc)$ for each $i=1,2,3$, from \eqref{TranslatesStrict} and thus it is above $$\max_{i=1,2,3} h_i(\xi_i\cdot(x-y) - mc)$$
for $y\in\tilde{T^e}(x,n;n +m)$.
Note that each such triangle $\tilde{T}^e$ has a type of buffer region of 
good boxes with density $\geq \alpha$ inside of $\tilde{T}^e$ that is distance $d(k)$ 
away from the edges of the triangle, as long as there are no overlapping triangles. With inward 
moving edges, once the triangle has inscribed circle of radius $d(k)$ or less, all boxes in 
the triangle are good. Any box that intersects both $\tilde{T}^e$ and $\tilde{A}_{n-1}^c$ is 
expected to have density of particles $> \alpha$ at time $n-1$.
We now define the errors.
\begin{enumerate}[ Type I Error:]
\item Such an error occurs if the density of a 
box in $\tilde{A}^c_{n}$ suddenly falls below $\alpha$ at time $n$, but the density in that box 
and all boxes within $d(k)$ of that box was $\geq \alpha$ at time $n-1$.
Suppose that $$B^L(x)\cap \tilde{A}_{n-1} = \varnothing.$$ 
Thus, $B^L(x)$ is good at time $n-1$ and all other boxes within distance $d(k)$ 
are also good at $n-1$.  If $B^L(x)$ becomes 
a low-density box at time $n$, this spontaneous 
error will be considered a {\it Type I} error. Note that if two (or more) boxes 
become low-density boxes at time $n$ that are within distance $d(k)$, 
two (or more) Type I errors occur and the two (or more) triangles that are created 
at time $n$ will certainly overlap, thus giving rise to overlap regions.
\item Such errors may occur when 
there is (are) an already low-density box(es) in 
a region of the plane but with density above 
appropriate translates of $h_{i}$'s. That is, assume that 
$$B^L(x)\cap \tilde{A}_{n-1} \neq \varnothing.$$ 
\end{enumerate}
To motivate the definition of a type II error that will follow and explain \eqref{translates},
 note that if a type I error occurs at $x\in\mathbb{R}^2$ at time $k$, 
a triangular region $\tilde{T}^e(x,k;k)$ is formed. If $B^L(x)$ is the only low-
density box in $\tilde{T}^e(x,k;k)$, then for each $y\in \tilde{T}^e(x,k;n)$, the density of $B^L(y)$ is greater than $$h(y,k,n) :=\max_{i=1,2,3} h_i(\xi_i\cdot(x-y) - (n-k)c).$$
In the above case, when no overlap regions are present, 
a type II error occurs if $y\in\tilde{T}^e(x,k;n)$ and the density 
of $B^L(y)$ at time $n$ falls below $h(y,k,n)$.

For what follows, a triangular overlap region centered at $x\in\mathbb{R}^2$ created 
at time $k$ will be denoted $\tilde{T}^o(x,k;k)$. If a type II error occurs, an overlap region(s)
 will certainly be created.
If $B^L(y)$ is contained in more than one triangular region at time $n$,
 consider the maximal collection of triangular regions such that 
\begin{align*}
 B^L(y)\subset \bigcap_{j\in J}\tilde{T}_j(y_j,n_j;n),
\end{align*}
where $J$ is a countable index set and $T_j(y_j,n_j;n)$ is a triangular region centered at $y_j\in\mathbb{Z}^2/L$ that was created at time $n_j\in\mathbb{N}$ and can either be of the type $\tilde{T}^e$ or of the type $\tilde{T}^o.$

Then the results of Bramson and Gray imply that there exists a region $\tilde{T}(y,m;n)$ centered at $y$ and created at time $m$ such that
\begin{align*}\label{IntersectionTypeII}
 \bigcap_{j\in J}\tilde{T}_j(y_j,n_j;t) \subset \tilde{T}(z,m;t) \text{ for all times } t\in [m,n],
\end{align*}
and so $B^L(y)\subset \tilde{T}(z,m;n)$.
\begin{enumerate}[{Case} i:]
\item If $\bigcap_{j\in J}\tilde{T}_j(y_j,n_j;n)$ contains only regions 
of the type $\tilde{T}^o$ and no regions of the type $\tilde{T}^e$,
 no type II error occurs at $x$. This is because no ``recent'' error has occurred within $d(k)$ of $y$, 
 and all boxes have density $>\alpha$. So, if an error were to occur at $y$, it would be of type I.
\item Suppose $\bigcap_{j\in J}\tilde{T}_j(y_j,n_j;n)$ contains at least one region 
of the type $\tilde{T}^e$. In this case, a type II error may occur at $x$. 
\end{enumerate}
For $i=1,2,3$, let $H_{i,j}(n)$ be the halfspace in $\mathbb{R}^2$ containing $\tilde{T}(y_j,n_j;n)$ and with boundary containing the side of the triangle perpendicular to $\xi_i$. 

Furthermore, the $i^{th}$ edge and corresponding halfspace containing $\tilde{T}$ and its $i^{th}$ edge at time $t$, $H_i(t)$ either:
(i) moves outward with rate $b$, or
(ii) moves inward with rate $c$ and is such that 
$$
\bigcup_{j\in J}H_{i,j}(t) \subset H_i(t). 
$$
If $H_i(n)$ is moving outward with rate $b$ for each $i=1,2,3$, no type II error occurs 
at $x$. 
If $H_i(n)$ is \emph{not} moving outward with rate $b$ for at least one of  $i=1,2,3$,
 let $E\subset\{1,2,3\}$ be the set of directions $i$ for which $H_i(n)$ is \emph{not} moving outward with rate $b$.
 Let 
$$h_{E,x}(n):=\max_{e\in E}\inf_{j\in J} h_{e}(\xi_e\cdot(x - y_j) - (n-m_{j,e})c),$$
where we define $m_{j,e}$ as follows: 
\begin{enumerate}[i)]
\item $n_j$, the time that the region $\tilde{T}(y_j,n_j;n)$ is created,  if $\tilde{T}(y_j,n_j;n)$ is a region of the form $\tilde{T}^e$,
\item the least integer greater than or equal to the time when $H_e(y_j,n_j,n)$ is not moving outward, if $\tilde{T}(y_j,n_j;n)$ is a region of the form $\tilde{T}^o$.
\end{enumerate}
A type II error occurs 
at $x$ at time $n$ if the density of $B^L(x)$ falls below $h_{E,x}(n)$. Note that in a given box $B^L(x)$, only one error may occur at a particular time $n$, type I or type II, not both.

\subsection{Probability of an error}
First consider an upper bound on the probability that a type I error occurs. Let $x \in \mathbb{Z}^2/L$ and $B^{L}(x)$ the box containing $x$. Then if the box is good, the total number of particles in it is $$S_{n-1}^{L}(x) > \alpha \cdot m,$$ where $m = L^{2-2\gamma}$ is the number of sites in $B(x)$.
From previous calculations,
\[ \mathbb{E}\left( S_{n}^{L}(x) \right) = \sum_{y \in B^{L}(x)} Q\big[u_{n-1}^{L}\big](y) \]  
and
\begin{align*}\mathbb{E}&\left[ S_{n}^{L}(x)\: \big\vert\: B(x) \text{ is good at time } n-1 \right] =  \sum_{y \in B^{L}(x)} Q\big[u_{n-1}^{L}\big](y) \geq Q\left(\alpha\right) \cdot m > \alpha\cdot m \end{align*}
We also have $$\text{Var}\left( S_{n}^{L}(x) \right) \leq C\cdot m.$$ 
Now fix any $\delta_1 \in (0,Q(\alpha) - \alpha)$.
By Chebyshev's inequality, 
\begin{align*}
P_{\hat{\xi}_{n-1}^{L}}\left(\left| S_{n}^{L}(x) - \sum_{y \in B^{L}(x)} Q(u_{n-1}^{L})(y)\right| \geq \delta_1 m \right) \leq \frac{Cm}{\delta_1^2 m^2} = \frac{C_1}{m} = C_1L^{2\gamma -2} := \epsilon_1
\end{align*}
This gives an upper bound on the probability of a good box spontaneously going bad, for a good box has density of particles $> \alpha$ and is therefore expected at the next time step to have density of particles $> Q(\alpha)$, so $$Q(\alpha) - \delta_1 > \alpha.$$ As $L\to\infty$, $\epsilon_1\to 0$.

For an upper bound on the probability of a type II error, choose $\delta_2 = \min_{i=1,2,3}\delta_{2,i} > 0$, with
$$\delta_{2,i}\in(0,\inf_{s\in[r,0]}\{Q[f_{n,i}\left(x\cdot\xi_i + s + c\right)](0,0)-h_i(s)\}).$$ 
By \eqref{TranslatesStrict}, $$\inf_{s\in[r,0]}\{Q[f_{n,i}\left(x\cdot\xi_i + s + c\right)](0,0)-h_i(s)\} > 0.$$
By Chebyshev's inequality, 
\begin{align*}
P_{\hat{\xi}_{n-1}^{L}}\left(\left| S_{n}^{L}(x) - \sum_{y \in B^{L}(x)} Q(u_{n-1}^{L})(y)\right| \geq \delta_2 m \right) \leq \frac{Cm}{\delta_2^2 m^2} = \frac{C_2}{m} = C_2L^{2\gamma -2} := \epsilon_2
\end{align*}
Thus, given the configuration at time $n-1$, the probability of an error occurring at time $n$, type I or type II, is bounded above by 
$\epsilon = \max\{\epsilon_1,\epsilon_2\}$,
with $\epsilon \to 0$ as $L\to\infty$.

\subsection{Point process of errors}

Next, we describe the point process, $\mathcal{P}$ on $\mathbb{R}^2\times[0,\infty)$ 
used in the comparison process. It is derived from the quadratic contact process 
on $\mathbb{Z}^2/L$. For each type I or type II error that occurs in the contact 
process, there is a single corresponding point in $\mathcal{P}$. If the error 
occurs in $B^L(x)$ at time $n$, let $(y,t)\in\mathcal{P}$ where 
$y\in x^* + [0,L^{-\gamma})^2$ is a single point in the box, chosen 
uniformly and at random from $x^* + [0,L^{-\gamma})^2$ and $t$ is 
chosen uniformly and at random from $[n,n+1)$.

 Although $\mathcal{P}$ is not quite a Poisson point process, it 
 shares two key properties with the Poisson process, the only two 
 used in Bramson and Gray's proof, that are sufficient to demonstrate 
 a nontrivial stationary distribution for large enough $L$. 
 Namely (see \citep[2-1 and 2-2]{BG-91}),
\begin{align}\label{P1}
 P\left(\mid B\cap\mathcal{P} \mid \geq 2 \right) = O\left(\lambda(B)^2 \right)
\end{align}
as $\lambda(B)\to 0$, where $\lambda$ is Lebesgue measure on $\mathbb{R}^2\times[0,\infty)$, for Borel sets $B$.
\eqref{P1} is satisfied by the quadratic contact error process $\mathcal{P}$, since at most one error can 
occur in $B^L(x)$ for any $x\in\mathbb{Z}^2/L$.

The second property is that for all small enough disjoint cubes $B_1,B_2,\ldots,B_m$ in $\mathbb{R}^2\times[0,\infty)$,
\begin{align}\label{P2}
 P\left(\bigcap_{j=1}^{m}\{B_j\cap\mathcal{P}\neq\varnothing\}\right) \leq \prod_{j=1}^{m} \left(2\epsilon L^{2\gamma}\lambda(B_j)\right).
\end{align}
If $B_j = b_j\times[s_j,t_j]$, where $b_j$ is a cube in $\mathbb{R}^2$, it is sufficient to assume that $$\lambda(b_j) < L^{-2\gamma}$$ for each $j$ and $t_j-s_j < 1$ for $j=1,2,\ldots,m$ (the time interval may also be open or half open) and that $b_j\subset\mathbb{R}^2$ overlaps at only one box $B^L(x)$. If $b_j$ overlaps more than one box $B^L(x)$, it can be split into a disjoint union of cubes each overlapping just one box. To see that \eqref{P2} is then satisfied, first consider the case when $n\in(t_j-s_j)$ for all $j=1,\ldots,m,$ and some $n\in\mathbb{N}$. 

Given the configuration of the quadratic contact process at time $n-1$, births at different sites at time $n$ are independent of each other. The same holds for deaths. For each $j$, $$P\left(B_j\cap\mathcal{P}\neq\varnothing\right) \leq \epsilon L^{2\gamma} \lambda(B_j).$$ Since the densities of particles in different boxes are independent, \eqref{P2} holds.

We still assume that  $$B_j = b_j\times[s_j,t_j],$$ where $b_j$ is a cube in $\mathbb{R}^2$, $\lambda(b_j) < L^{-2\gamma}$ for each $j$ and $t_j-s_j < 1$. If $n_j\in(t_j-s_j)$ for some $n_j\in\mathbb{N}$, split $B_j$ into two cubes: $$B_{j,1} = b_j\times[s_j,n_j)\quad \text{and}\quad B_{j,2} = b_j\times[n_j,t_j].$$ We note that 
\begin{align*}
 P\left(\{B_{j,1}\cap\mathcal{P}\neq\varnothing\}\cap\{B_{j,2}\cap\mathcal{P}\neq\varnothing\}\right) \leq 2\epsilon L^{2\gamma}\max\{\lambda(B_{j,1}),\lambda(B_{j,2})\} 
\end{align*}
So we can suppose this does not occur and let $n_j = \lfloor s_j \rfloor$. Also suppose that the cubes are ordered in a way that $$n_1\leq n_2\leq\ldots\leq n_m.$$ \eqref{P2} can be shown by induction on $m$. It clearly holds for $m=1$. Let $$A_j = \{B_j\cap\mathcal{P}\neq\varnothing\}.$$ Then
\begin{align*}
 P\left(\bigcap_{j=1}^{m} A_j\right) &= P\left(A_m \bigg| \bigcap_{j=1}^{m-1} A_j\right)P\left(\bigcap_{j=1}^{m-1} A_j\right)\leq \epsilon L^{2\gamma}\lambda(B_m)\prod_{j=1}^{m-1} \left(2\epsilon L^{2\gamma}\lambda(B_j)\right).
\end{align*}
The second factor comes from the induction assumption and the first factor is from the upper bound on the error probability.

\subsection{Comparison Result}

\begin{theorem}\label{ComparisonThm}
Let $\tilde{A}_n$ be the bad region of the contact process at time $n$ with death probability $\eta$, birth probability $\beta$, and finite offspring distribution kernel $k$ such that there is at least one direction with a positive wave speed. 

Let $A_t$ be the vacant region of the comparison process with point process $\mathcal{P}$, orientation vectors $n_i = \xi_i$, speeds $a_i = c$, and interaction rate $b = r$. Then the processes $A_n$ and $\tilde{A}_n$ can be jointly coupled so that 
\begin{align*}
\tilde{A}_n \subset A_n
\end{align*}
for all $n\in\{0,1,2,\ldots \}$.
Hence, for large enough $L$, the quadratic contact process has a nontrivial stationary distribution.
\end{theorem}

\begin{proof}

The two processes are coupled in the following way: if in the contact process, an error occurs at time $n$ in some box, place a single point in the cube $$Q(x) = B^L(x)\times [n,n+1) \in \mathbb{R}^2 \times [0,\infty)$$  uniformly at random. If no error occurs, leave the corresponding box empty in $\mathcal{P}$. 
The rates have been set up so that all boxes with density $<\alpha$ in the contact process are covered by some triangular region. A triangular region automatically covers type I and type II errors. All other low density boxes can only be in the vicinity of the type I and II errors: within distance $d(K)$ times the age of the error. The interaction rate $b$ ensures that if there is a large cluster of errors, the overlap/collision region grows faster than the surrounding bad boxes can spread (only by $d(k)$ units per single time step), and from the perimeter of the bad regions, the positive wave speeds ``propagate'' the high density boxes into the former bad regions. 

We use an induction argument. Run the contact process from an initial configuration with every site occupied: $\tilde{A}_0 \subset A_0$ since both processes begin with empty vacant region. Now, assuming $\tilde{A}_{n-1} \subset A_{n-1}$, we show that after one time step, we still have $\tilde{A}_n \subset A_n$, which will follow from our earlier definitions of the errors and parameters. 

Suppose that $x\in\tilde{A}_n$. This means that 
the density of particles in $B^L(x)$ at time $n$ is less than $\alpha$. If also $x\notin\tilde{A}_{n-1}$, all boxes within $d(k)$ of $x$ are good at time $n-1$, since any triangular region $\tilde{T}(y,m;n-1)$ has a ``buffer'' region within 
$d(k)$ of each edge in which the density of particles is expected to be $> \alpha$, by construction. Thus, for an error to occur at time $n$ in $B^L(x)$, 
it must be a type I error, and $B^L(x)$ is covered by a triangular vacant region centered at a uniformly selected point in $Q(x)$ in the comparison process, so $B^L(x)\in A_n$.

If $x\in\tilde{A}_n$ and $x\in\tilde{A}_{n-1}$, there are two possibilities:  
\begin{enumerate}[(i)]
\item\label{en:1} either the density in $B^L(x)$ at time $n-1$ was less than $\alpha$, or 
\item\label{en:2} the density in $B^L(x)$ at time $n-1$ was greater than or equal to $\alpha$.
\end{enumerate}
If~\eqref{en:1} holds, then $B(x;d(k)+c)\subset A_{n-1}$, since by definition of the errors, if the density of a box is below $\alpha$, it is 
contained in a triangular vacant region and so is the ball around it of radius at least $d(k) + c$. Thus, since a triangular vacant region shrinks by at most $c$ units in any direction in one time unit, $B^L(x)\subset\tilde{A}_{n}\subset A_{n}$.

If~\eqref{en:2} holds and if $x$ is within $d(k)$ of the nonvacant region at time $n-1$, at this distance by construction the expected density of $B^L(x)$ is $>\alpha$.
 Then a type II error occurs at $x$ at time $n$, so $B^L(x)\subset \tilde{A}_n\subset A_n$. If $x$ is not within $d(k)$ of the nonvacant region at time $n-1$, the density of $B^L(x)$ is either above $h_{E,x}(n-1)$ and above $h_{E,x}(n)$, in which case $x\in A_n$, or it is above $h_{E,x}(n-1)$ and below $h_{E,x}(n)$, in which case a type II error occurs at $x$.

 The comparison process $A_n$ has a nontrivial stationary distribution. Hence, so does the contact process.
\end{proof}

%%%%%%%%%%%%%%%%%%%%%%%%%%%%%%%%%

\bibliographystyle{elsarticle-harv}
\bibliography{research}
\end{document}